\documentclass[11pt,a4paper,reqno]{amsart}
\title[On the magnitude of a finite dimensional algebra]
 {On the magnitude
  of a finite dimensional algebra}
\date{\today}

\author{Joseph Chuang}
\address{JC: Centre for Mathematical Science\\ City Univ.\\
 London EC1V 0HB\\ U.K.}
\email{j.chuang@city.ac.uk}

\author{Alastair King}
\address{AK: Mathematical Sciences\\ Univ. of Bath\\
 Bath BA2 7AY\\ U.K.}
\email{a.d.king@bath.ac.uk}

\author{Tom Leinster}
\address{TL: School of Mathematics\\ Univ. of Edinburgh\\
 Edinburgh EH9 3FD\\ U.K.}
\email{tom.leinster@ed.ac.uk}
\usepackage{geometry} 
\usepackage{url,hyperref}
\usepackage{color}
\definecolor{darkblue}{rgb}{0,0,0.5}
\hypersetup{colorlinks,
linkcolor=darkblue,
citecolor=darkblue,
urlcolor=darkblue}

\usepackage{amssymb,amsthm,array}

\theoremstyle{plain}
\newtheorem{theorem}{Theorem}[section]

\theoremstyle{definition}

\newtheorem{example}[theorem]{Example}

\numberwithin{equation}{section}

\newcommand{\QQ}{\mathbb{Q}}
\newcommand{\ZZ}{\mathbb{Z}}
\newcommand{\tensor}{\otimes}
\newcommand{\Hom}{\operatorname{Hom}}
\newcommand{\Ext}{\operatorname{Ext}}
\newcommand{\End}{\operatorname{End}}

\newcommand{\class}[1]{\bigl[#1\bigr]}
\newcommand{\classl}[1]{[#1]}
\newcommand{\Grot}[1]{\operatorname{K}(#1)}
\newcommand{\DM}{Z}
\newcommand{\dm}{Z}
\newcommand{\CM}{C}
\newcommand{\cm}{C}
\newcommand{\cmi}{\overline{C}}
\newcommand{\dmi}{\overline{Z}}
\newcommand{\EuM}{E}
\newcommand{\eum}{E}
\newcommand{\Stil}{\widetilde{S}}

\newcommand{\cat}[1]{\mathbf{#1}}
\newcommand{\scat}[1]{\mathbf{#1}}
\newcommand{\fcat}[1]{\mathbf{#1}}
\newcommand{\mg}[1]{\left| #1 \right|}
\newcommand{\Mod}[1]{#1\text{-}\fcat{Mod}}
\newcommand{\IP}[1]{\fcat{IP}(#1)}
\newcommand{\Vect}{\fcat{Vect}}
\newcommand{\demph}[1]{\textbf{#1}}
\newcommand{\RR}{\mathbb{R}}
\newcommand{\NN}{\mathbb{N}}
\newcommand{\eqv}{\simeq}
\newcommand{\iso}{\cong}
\newcommand{\from}{\colon}
\newcommand{\op}{{\operatorname{op}}}

\begin{document}

\sloppy

\begin{abstract}
There is a general notion of the magnitude of an enriched category, defined
subject to hypotheses.  In topological and geometric contexts, magnitude is
already known to be closely related to classical invariants such as Euler
characteristic and dimension.  Here we establish its significance in an
algebraic context.  Specifically, in the representation theory of an
associative algebra~$A$, a central role is played by the indecomposable
projective $A$-modules, which form a category enriched in vector spaces.
We show that the magnitude of that category is a known homological
invariant of the algebra: writing $\chi_A$ for the Euler form of $A$ and
$S$ for the direct sum of the simple $A$-modules, it is $\chi_A(S, S)$.
\end{abstract}

\maketitle

\section{Introduction}

This paper is part of a large programme to define and investigate
cardinality-like invariants of mathematical objects.  Given a monoidal
category $\cat{V}$ together with a notion of the `size' $\mg{X}$ of each
object $X$ of $\cat{V}$, there arises automatically a notion of the `size'
or `magnitude' of each $\cat{V}$-category (subject to conditions).  Here we
apply this general method in the context of associative algebras.

More specifically, for any finite-dimensional algebra $A$, the category
$\IP{A}$ of indecomposable projective $A$-modules plays a central role
(discussed below) in the theory of representations of $A$.  This category
is enriched in finite-dimensional vector spaces, and, taking dimension as
the base notion of size, we can then consider the magnitude of $\IP{A}$.
We show that this is a known homological invariant of the original algebra
$A$.

Little algebra will be assumed on the reader's part; all the necessary
background is provided in Section~\ref{sec:background}.

The general definition of magnitude is as follows~\cite[\S 1.3]{LeMMS}.
Let $\cat{V}$ be a monoidal category equipped with a function
$\mg{\,\cdot\,}$ on its set of objects (taking values in a semiring, say).
Let $\scat{A}$ be a $\cat{V}$-category with finitely many objects.  Denote
by $\DM_\scat{A} = (\dm_{ab})$ the square matrix whose rows and columns are
indexed by the objects of $\scat{A}$, and with entries
\begin{equation}
\dm_{ab} = \mg{\scat{A}(a, b)}
\end{equation}
($a, b \in \scat{A}$).  If $\DM_\scat{A}$ is invertible, the
\demph{magnitude} $\mg{\scat{A}}$ of $\scat{A}$ is defined to be the sum of
all the entries of $\DM_\scat{A}^{-1}$.

Since $\DM_\scat{A}$ need not be invertible, magnitude is not defined for
every $\scat{A}$.  But where magnitude \emph{is} defined, we may harmlessly
extend the definition by equivalence, setting $\mg{\scat{A}} =
\mg{\scat{B}}$ whenever $\scat{A}$ and $\scat{B}$ are equivalent
$\cat{V}$-categories such that $\scat{B}$ has finitely many objects and
$\DM_\scat{B}$ is invertible.  (There is no problem of consistency, since if
$\scat{A}$ and $\scat{B}$ are equivalent and both $\DM_\scat{A}$ and
$\DM_\scat{B}$ are invertible then both $\scat{A}$ and $\scat{B}$ are
skeletal---that is, isomorphic objects are equal---and so $\scat{A}$ and
$\scat{B}$ are isomorphic.)

Unmotivated as this definition may seem, multiple theorems attest that
magnitude is the canonical notion of the size of an enriched category.  For
example, take $\cat{V}$ to be the category of finite sets and $\mg{X}$ to
be the cardinality of a finite set $X$.  Then we obtain a notion of the
magnitude of a finite category.  In this context, magnitude is also called
Euler characteristic~\cite{LeECC}, for the following reason.  Recall that
every small category $\scat{A}$ gives rise to a topological space
$B\scat{A}$, its classifying space or geometric realisation.
Proposition~2.11 of~\cite{LeECC} states that under finiteness hypotheses,
\begin{equation}
\mg{\scat{A}} = \chi(B\scat{A}).
\end{equation}
Thus, the Euler characteristic of a category has a similar status to group
(co)homology: it is defined combinatorially, but agrees with the
topological notion when one passes to the classifying space.

For another example, let $\cat{V}$ be the ordered set $([0, \infty], \geq)$
with addition as the monoidal structure, so that metric spaces can be
viewed as $\cat{V}$-categories~\cite{La}.  For $x \in [0, \infty]$, put
$\mg{x} = e^{-x}$.  (The virtue of this choice is that $\mg{x \otimes y} =
\mg{x} \mg{y}$.)  Then we obtain a notion of the magnitude of a finite
metric space.  This extends naturally to a large class of compact metric
spaces~\cite{LeMMS,MePDMS,MeMDC}.  The magnitude of a compact subset of
$\RR^n$ is always well-defined, and is closely related to classical
quantities of geometric measure.  For example, a theorem of Meckes
\cite[Corollary 7.4]{MeMDC} shows that Minkowski dimension can be
recovered from magnitude, and conjectures of Leinster and
Willerton~\cite{LeWi} state that magnitude also determines invariants
such as volume and surface area.  

Here we study the case where $\cat{V}$ is the category of
finite-dimensional vector spaces and $\mg{X} = \dim X$.  We then obtain a
notion of the magnitude of a linear (that is, $\cat{V}$-enriched) category.
Our main theorem is this:

\begin{theorem}
\label{thm:main}
Let $A$ be an algebra of finite dimension and finite global dimension over
an algebraically closed field.  Write $\IP{A}$ for the linear category of
indecomposable projective $A$-modules, $(S_i)_{i \in I}$ for
representatives of the isomorphism classes of simple $A$-modules, and $S =
\bigoplus_{i \in I} S_i$.  Then
\begin{equation}
\label{eq:main}
\mg{\IP{A}} 
=
\sum_{n = 0}^\infty (-1)^n \dim \Ext_A^n(S, S).
\end{equation}
\end{theorem}

We now explain the context of this result; background can be found in the
next section.  

Any associative algebra $A$ gives rise to several linear categories,
including the category of all $A$-modules and the one-object category
corresponding to $A$ itself (which trivially has magnitude $1/\dim A$).
But it also gives rise to the category $\IP{A}$ of indecomposable
projective $A$-modules, whose main significance is that its representation
theory is the same as that of $A$:
\begin{equation}
\label{eq:morita}
\begin{array}{ccc}
\Mod{A}         &\eqv           &[\IP{A}^\op, \Vect]    \\
M               &\mapsto        &\Hom_A(-, M)
\end{array}
\end{equation}
where the right-hand side is the category of contravariant linear functors
from $\IP{A}$ to vector spaces.  In other words, $\IP{A}^\op$ and the
one-object linear category $A$ are Morita equivalent.


The Krull--Schmidt theorem states that every finitely generated $A$-module
can be expressed as a direct sum of indecomposable modules, in an
essentially unique way.  It implies that the $A$-module $A$ is a direct sum
of indecomposable \emph{projective} modules, and that, moreover, 
every indecomposable projective appears at least once
in this sum.  Thus, the indecomposable projectives are the `atoms' of
$A$, in the sense of being its constituent parts.

This explains the equivalence~\eqref{eq:morita}.  The absolute colimits in
linear categories are the finite direct sums and idempotent splittings
(that is, direct summands).  Every finitely generated projective module is
a direct sum of indecomposable projectives, so the category of finitely
generated projectives is the Cauchy completion of $\IP{A}$.  On the other
hand, every finitely generated projective is a direct summand of a direct
sum of copies of the $A$-module $A$, so the category of finitely generated
projectives is also the Cauchy completion of the one-object category
$A^\op$.  Hence $\IP{A}$ and $A^\op$ have the same Cauchy completion, and
are therefore Morita equivalent.

The simple modules, too, can be thought of as `atomic' in a different
sense.  A simple module need not be indecomposable projective, nor vice
versa.  However, the two conditions are closely related: as recounted in
Section~\ref{sec:background}, there is a canonical bijection between the
isomorphism classes of simple modules and the isomorphism classes of
indecomposable projectives.


The condition that $A$ has finite global dimension guarantees that the sum
in~\eqref{eq:main} has only finitely many nonzero terms.  The condition
that $A$ has finite dimension guarantees that the linear category $\IP{A}$
is equivalent to one with finitely many objects and finite-dimensional
hom-spaces, as we shall see.  This is a necessary condition in order for
the magnitude of $\IP{A}$ to be defined.  It is not a sufficient condition,
but part of the statement of Theorem~\ref{thm:main} is that $\mg{\IP{A}}$
\emph{is} defined.

Theorem~\ref{thm:main} was first noted by Catharina Stroppel under the
additional hypothesis that $A$ is a Koszul algebra (personal communication,
2009).  We observe here that the Koszul assumption is unnecessary.

\section{Algebraic background}
\label{sec:background}

Here we assemble all the facts that we will need in order to state and
prove the main theorem.  General references for this section
are~\cite[Chapter~I]{SY} and \cite[Chapter~1]{Be}.

Throughout this note, $K$ denotes a field and $A$ a finite-dimensional
$K$-algebra (unital, but not necessarily commutative).  `Module' will
mean left $A$-module.  Since $A$ is finite-dimensional, a module is
finitely generated over $A$ if and only if it is finite-dimensional over
$K$.

\subsection*{Simple and indecomposable projective modules}

Details for this part can be found in~\cite{LeBB}, as well as in the
general references above.

A nonzero module is \demph{simple} if it has no nontrivial submodule, and
\demph{indecomposable} if it has no nontrivial direct summand.  There is a
canonical bijection between the isomorphism classes of simple modules $S$
and the isomorphism classes of indecomposable projective modules $P$, with
$S$ corresponding to $P$ if and only if $S$ is a quotient of $P$.  (It is
not an equivalence of categories.)


Choose representatives $(S_i)_{i \in I}$ of the isomorphism classes of
simple modules and $(P_i)_{i \in I}$ of the isomorphism classes of
indecomposable projective modules, with $S_i$ a quotient of $P_i$.

Modules of both types are finitely generated (indeed, cyclic), so each
vector space $\Hom_A(P_i, P_j)$ is finite-dimensional.  Moreover, one can
use either the Jordan--H\"older theorem or the Krull--Schmidt theorem to
show that $I$ is finite.  Denote by $\IP{A}$ the category of indecomposable
projective $A$-modules and all homomorphisms between them, which is a
$K$-linear category.  Then $\IP{A}$ has finite-dimensional hom-spaces and
only finitely many isomorphism classes of objects.

We have $\Hom_A(P_i, S_j) = 0$ when $i \neq j$, since any homomorphism into
a simple module is zero or surjective.  It can be shown that $\Hom_A(P_i,
S_i) \iso \End_A(S_i)$ as vector spaces.  This is a skew field, isomorphic
to $K$ if $K$ is algebraically closed.

\subsection*{Homological algebra}

For each $n \geq 0$, there is a functor 
\begin{equation}
\Ext_A^n \from \Mod{A}^\op \times \Mod{A} \to \Vect.  
\end{equation}
One can characterise $\Ext_A^n(X, -)$ as the $n$th right derived functor of
$\Hom_A(X, -)$, and $\Ext_A^n(-, Y)$ as the $n$th right derived functor of
$\Hom_A(-, Y)$, but we will only need the following consequences of these
characterisations.

First, $\Ext_A^0 = \Hom_A$.  Second, if $P$ is projective then $\Ext_A^n(P, -)
= 0$ for all $n > 0$.  Third, $\Ext_A^n$ preserves finite direct sums in
each argument.  Fourth, $\Ext_A^n(X, Y)$ is finite-dimensional if both $X$
and $Y$ are.  Finally, given any $A$-module $V$ and short exact sequence
\begin{equation}
\label{eq:ses}
0 \to W \to X \to Y \to 0,
\end{equation}
there is an induced long exact sequence
\begin{align}
0       &
\to      
\Ext_A^0(V, W) \to \Ext_A^0(V, X) \to \Ext_A^0(V, Y)    
\nonumber \\
&
\to    
\Ext_A^1(V, W) \to \Ext_A^1(V, X) \to 
\cdots,
\label{eq:les}
\end{align}
and dually a long exact sequence $0 \to \Ext_A^0(Y, V) \to \cdots$.

Assume henceforth that $A$ has \demph{finite global
  dimension}~\cite[Chapter~4]{We}.  This means that there exists $N \in
\NN$ such that every $A$-module $X$ has a projective resolution of the form
\begin{equation}
\label{eq:proj-reso}
0 \to Q_N \to \cdots \to Q_1 \to X \to 0.
\end{equation}
When $X$ is finite-dimensional, the projective modules $Q_i$ can be
chosen to be finite-dimensional too.

A condition equivalent to finite global dimension is that $\Ext_A^n = 0$
for all $n \gg 0$.  For finite-dimensional $A$-modules $X$ and $Y$, we may
therefore define
\begin{equation}
\chi_A(X, Y) = \sum_{n = 0}^\infty (-1)^n \dim \Ext_A^n(X, Y) \in \ZZ
\end{equation}
(a finite sum).  This $\chi_A$ is the \demph{Euler form} of $A$.  We have
$\chi_A(\bigoplus_r X_r, -) = \sum_r \chi_A(X_r, -)$ for any finite family
$(X_r)$ of modules, and similarly in the second argument.  Moreover, the
observations above imply that
\begin{equation}
\label{eq:chi-Z}
\chi_A(P_i, P_j) = \dim\Hom_A(P_i, P_j)
\end{equation}
for all $i, j \in I$, and that 
\begin{equation}
\label{eq:pairing}
\chi_A(P_i, S_j) 
=
\begin{cases}
\dim \End_A(S_j)        &\text{if } i = j,      \\
0                       &\text{if } i \neq j.
\end{cases}
\end{equation}
When $K$ is algebraically closed, $\chi_A(P_i, S_j)$ is therefore just the
Kronecker delta $\delta_{ij}$.

\subsection*{Grothendieck group}

The \demph{Grothendieck group} $\Grot{A}$ is the abelian group generated by
the finite-dimensional $A$-modules, subject to the relation $X = W + Y$ for
each short exact sequence~\eqref{eq:ses} of finite-dimensional modules.
Writing $\classl{X}$ for the class of $X$ in $\Grot{A}$, one easily deduces
that, more generally, $\sum_r (-1)^r \class{X_r} = 0$ for any exact
sequence
\begin{equation}
0 \to X_1 \to \cdots \to X_n \to 0.
\end{equation}
For example, take a short exact sequence~\eqref{eq:ses} and a
finite-dimensional module $V$.  The resulting long exact
sequence~\eqref{eq:les} has only finitely many nonzero terms (since $A$ has
finite global dimension), so the alternating sum of the dimensions of these
terms is $0$, giving $\chi_A(V, X) = \chi_A(V, W) + \chi_A(V, Y)$.  The
same holds with the arguments reversed.  Thus, $\chi_A$ defines a
$\ZZ$-bilinear map $\Grot{A} \times \Grot{A} \to \ZZ$.

We now show that $\Grot{A}$ is free as a $\ZZ$-module, and in fact has two
canonical bases.

First, the family $\bigl(\class{S_i}\bigr)_{i \in I}$ generates the group
$\Grot{A}$.  Indeed, for any finite-dimensional $A$-module $X$, we may take
a composition series
\begin{equation}
0 = X_n < \cdots < X_1 < X_0 = X,
\end{equation}
and then $\classl{X} = \sum_{r = 1}^n \class{X_{r - 1}/X_r}$.

Second, the family $\bigl(\class{P_i}\bigr)_{i \in I}$ generates
$\Grot{A}$.  Given a finite-dimensional $A$-module $X$, we may take a
resolution~\eqref{eq:proj-reso} by finite-dimensional projective modules,
and then $\classl{X} = \sum_{r = 1}^N (-1)^{r + 1} \class{Q_r}$.  On the
other hand, each $Q_r$ is a finite direct sum of indecomposable submodules,
which are projective since $Q_r$ is.

Finally, both $\bigl(\class{S_i}\bigr)$ and $\bigl(\class{P_i}\bigr)$
\emph{freely} generate the abelian group $\Grot{A}$.  This follows
from~\eqref{eq:pairing} and the $\ZZ$-bilinearity of $\chi_A$.

\section{The result}

Recall our standing conventions: $A$ is an algebra of finite dimension and
finite global dimension, over a field $K$ which we now assume to be
algebraically closed.  We continue to write $(P_i)_{i \in I}$ for
representatives of the isomorphism classes of indecomposable projective
$A$-modules, and similarly $(S_i)_{i \in I}$ for the simple modules, with
$S_i$ a quotient of $P_i$.

The linear category $\IP{A}$ of indecomposable projective $A$-modules is
equivalent to its full subcategory with objects $P_i$ ($i \in I$).  Write
$\DM_A = (\dm_{ij})_{i, j \in I}$ for the matrix of this finite linear
category, so that $\dm_{ij} = \dim\Hom_A(P_i, P_j)$.

We will derive our main result, Theorem~\ref{thm:main}, from the following
basic theorem.  (See e.g.~\cite[Proposition~III.3.13(a)]{ASS} for an
essentially equivalent formulation.)  It implies, in particular, that the
matrix $\DM_A$ is invertible over the integers.

\begin{theorem}
\label{thm:inverse}
The inverse of the matrix $\DM_A$ is the `Euler matrix' $\EuM_A =
(\eum_{ij})_{i, j \in I}$, given by $\eum_{ij} = \chi_A(S_j, S_i)$.
\end{theorem}

\begin{proof}
Since $\bigl(\class{P_i}\bigr)_{i \in I}$ and $\bigl(\class{S_i}\bigr)_{i
  \in I}$ are both bases for the $\ZZ$-module $\Grot{A}$, there is an
invertible matrix $\CM_A = (\cm_{ij})_{i, j \in I}$ over $\ZZ$ such that,
writing $C_A^{-1} = (\cmi_{ij})$,
\begin{align}
\label{eq:base-change-P}
\class{P_j}     &
= \sum_{k \in I} \cm_{kj} \class{S_k},  \\
\label{eq:base-change-S}
\class{S_j}     &
= \sum_{k \in I} \cmi_{kj} \class{P_k}
\end{align}
for all $j \in I$.  Since $K$ is algebraically closed,
equation~\eqref{eq:pairing} states that 
$\chi_A(P_i, S_j) = \delta_{ij}$.  Applying $\chi_A(P_i, -)$ to each side
of~\eqref{eq:base-change-P} therefore gives $\chi_A(P_i, P_j) = C_{ij}$,
which by~\eqref{eq:chi-Z} is equivalent to $\dm_{ij} = C_{ij}$.  On the other
hand, applying $\chi_A(-, S_i)$ to each side of~\eqref{eq:base-change-S}
gives $\eum_{ij} = \cmi_{ij}$.  Hence $\DM_A = \CM_A$ and $\EuM_A =
\CM_A^{-1}$.
  \end{proof}

The matrix $\CM_A = \DM_A$ is known as the \demph{Cartan matrix} of $A$
(\cite{BFVZ}, \cite[\S5]{Ei}, \cite{Fu}).  Explicitly, $\cm_{ij}$ is the
multiplicity of $S_i$ as a composition factor of $P_j$.

We now deduce Theorem~\ref{thm:main}.  By definition, $\mg{\IP{A}}$ is the
sum of the entries of $\DM_A^{-1}$.  Hence by Theorem~\ref{thm:inverse} and
the $\ZZ$-bilinearity of $\chi_A$,
\begin{equation}
\mg{\IP{A}} 
= 
\sum_{i, j \in I} \chi_A(S_j, S_i) 
=
\chi_A \biggl( \bigoplus_{j \in I} S_j, \bigoplus_{i \in I} S_i \biggr)
=
\chi_A(S, S),
\end{equation}
completing the proof.


\begin{example}
\label{eg:acyclic}
Let $Q$ be a finite acyclic quiver (directed graph).  Then $Q$ consists of
a finite set $I$ of vertices together with, for each $i, j \in I$, a finite
set $Q(i, j)$ of arrows from $i$ to $j$.  The \demph{path algebra} $A$ of
$Q$ is defined as follows.  As a vector space, it is generated by the paths
in $Q$, including the zero-length path $e_i$ on each vertex $i$.
Multiplication is concatenation of paths where that is defined, and zero
otherwise.  We write multiplication in the same order as composition, so
that if $\alpha$ is a path from $i$ to $j$ and $\beta$ is a path from $j$
to $k$ then $\beta\alpha$ is a path from $i$ to $k$.  The identity is
$\sum_{i \in I} e_i$.  That $Q$ is finite and acyclic guarantees that $A$
is of finite dimension and finite global dimension.

Path algebras of quivers are very well-understood
(e.g.~\cite[Chapter~I]{SY}).  The simple and indecomposable projective
$A$-modules are indexed by the vertex-set $I$.  The indecomposable
projective module $P_i$ corresponding to vertex $i$ is the submodule of the
$A$-module $A$ spanned by the paths beginning at $i$.  It has a unique
maximal submodule $N_i$, spanned by the paths of nonzero length beginning
at $i$, and the corresponding simple module $S_i = P_i/N_i$ is
one-dimensional.

Using the facts listed in Section~\ref{sec:background}, we can compute the
Euler form of $A$.  For each $i, j \in I$, the short exact sequence
\begin{equation}
0 \to N_i \to P_i \to S_i \to 0
\end{equation}
gives rise to a long exact sequence
\begin{equation}
\label{eq:quiver-les}
0 \to \Ext_A^0(S_i, S_j) \to \Ext_A^0(P_i, S_j) \to \Ext_A^0(N_i, S_j)
\to \cdots.
\end{equation}
Observing that $N_i = \bigoplus_{k \in I} P_k^{Q(i, k)}$, we deduce
from~\eqref{eq:quiver-les} that
\begin{equation}
\Ext_A^n(S_i, S_j)
=
\begin{cases}
K^{\delta_{ij}} &\text{if } n = 0,      \\
K^{Q(i, j)}     &\text{if } n = 1,      \\
0               &\text{if } n \geq 2.
\end{cases}
\end{equation}
Hence, writing $E = \coprod_{i, j \in Q} Q(i, j)$ for the set of arrows of
$Q$, 
\begin{equation}
\Ext_A^n(S, S)
=
\begin{cases}
K^{|I|}         &\text{if } n = 0,      \\
K^{|E|}         &\text{if } n = 1,      \\
0               &\text{if } n \geq 2.
\end{cases}
\end{equation}
It follows that $\chi_A(S, S) = |I| - |E|$, which is the Euler
characteristic (in the elementary sense) of the quiver $Q$.

On the other hand, each path from vertex $j$ to vertex $i$ induces a
homomorphism $P_i \to P_j$ by composition, and in fact every homomorphism
$P_i \to P_j$ is a unique linear combination of homomorphisms of this form.
Hence $\dm_{ij}$ is the number of paths from $j$ to $i$ in $Q$.

So in the case at hand, Theorem~\ref{thm:main} states that if we take an
acyclic quiver $Q$, form the matrix whose $(i, j)$-entry is the number of
paths from $j$ to $i$, invert this matrix, and sum its entries, the result
is equal to the Euler characteristic of $Q$.  This was also shown directly
as Proposition~2.10 of~\cite{LeECC}.
\end{example}

\section{Some remarks}

\subsection*{Arbitrary base fields}
 
The assumption that the base field is algebraically closed 
is needed for the simple form of the duality formula, $\chi_A(P_i, S_j) =
\delta_{ij}$.  Otherwise, equation~\eqref{eq:pairing} only gives
\begin{equation}\label{eq:duality-d}
\chi_A(P_i, S_j) 
=
\begin{cases} 
d_j     &\text{if } i = j, \\ 
0       &\text{if } i \neq j,
\end{cases}
\end{equation}
where $d_j=\dim \End_A(S_j)$.  Then, applying $\chi_A(P_i, -)$ to
\eqref{eq:base-change-P} yields $\dm_{ij} = d_i \cm_{ij}$, while applying
$\chi_A(-, S_i)$ to \eqref{eq:base-change-S} yields $E_{ij} = d_i
\cmi_{ij}$.  Therefore, writing $\DM_A^{-1} = (\dmi_{ij})$, we get
\begin{equation}\label{eq:MEgen}
   \dmi_{ij} = d_i^{-1} E_{ij} d_j^{-1},
\end{equation}
which generalises Theorem~\ref{thm:inverse}.
We can then sum \eqref{eq:MEgen} to generalise Theorem~\ref{thm:main}
as follows:
\begin{equation}
\label{eq:mag=chi-v2}
\mg{\IP{A}} = \chi_A(\Stil,\Stil),
\end{equation} 
where $\Stil = \bigoplus_{i\in I} d_i^{-1} S_i$,
which may be regarded as a formal module
or we may note that \eqref{eq:mag=chi-v2} only depends on the class
$\class{\Stil} = \sum_{i\in I} d_i^{-1} \class{S_i} 
\in \Grot{A}\tensor_\ZZ \QQ$.

\subsection*{The determinant of the Cartan matrix}

The fact that, when $A$ has finite global dimension,
the Cartan matrix $\CM_A$ is invertible over $\ZZ$
or, equivalently, is unimodular, i.e.\ $\det C_A =\pm 1$,  
is an old observation of Eilenberg~\cite[\S 5]{Ei}.
On the other hand, the `Cartan determinant conjecture' that, 
in fact, $\det C_A = 1$ is still unsolved in general, 
although it is confirmed in many cases; see~\cite{Fu} for a survey.

An easy example is when $A$ is (Morita equivalent to) a quotient of the
path algebra of an acyclic quiver, in which case $A$ is necessarily finite
dimensional and of finite global dimension. In this case $\CM_A = \DM_A$ can
be made upper triangular with $1$s on the diagonal, so it certainly has
$\det\CM_A =1$.
As another example, Zacharia~\cite{Za} showed that the conjecture
holds whenever $A$ has global dimension 2.

It is not hard to give an example of an algebra $A$
for which $\CM_A = \DM_A$ is not even invertible over $\QQ$: 
e.g.\ the quiver algebra given by a single $n$-cycle,
with all paths of length $n$ set to 0, has $\cm_{ij}=\dm_{ij}=1$
for all $i, j\in I$.
Inevitably, this algebra does not have finite global dimension.

In this example, it is in fact still possible~\cite[\S1]{LeMMS} to define the
magnitude of $\IP{A}$, and indeed $\mg{\IP{A}} = 1$.  However, it is less
clear how one might find a homological interpretation of this.

\subsection*{Acknowledgements} We thank Iain Gordon, Catharina Stroppel,
Peter Webb and Michael Wemyss for helpful discussions.


%
%
%
%
%
%
%
\end{document}